\documentclass{scrartcl}
\usepackage[latin9]{inputenc}
\usepackage{a4}
\usepackage{amsfonts}
\usepackage{amssymb}
\usepackage{amsmath}
\usepackage{amsthm}
\usepackage{makeidx}
\usepackage{pdflscape}
\usepackage{setspace}
\usepackage{multirow}
\usepackage{tabularx}
\usepackage{float}
\makeindex
\usepackage{nomencl}
\usepackage{geometry}
\usepackage{extarrows}
\usepackage{changepage}
\usepackage{chngcntr} 
\usepackage[pdftex]{graphicx}
\usepackage{tikz}
\usepackage[toc,page]{appendix}
\providecommand{\keywords}[1]{\textbf{\textit{Keywords: }} #1}
\theoremstyle{plain}
\newtheorem{satz}{Theorem}\numberwithin{satz}{section} %[section]
\newtheorem{lemma}[satz]{Lemma}%\numberwithin{satz}{section} 
\newtheorem{prop}[satz]{Proposition}%\numberwithin{satz}{section} 
\newtheorem{kor}[satz]{Corollary}%\numberwithin{satz}{section} 
\theoremstyle{definition}
\newtheorem{defi}{Definition}\numberwithin{defi}{section} 
\theoremstyle{remark}
\numberwithin{conj}{section}
\newtheorem{remark}{Remark}\numberwithin{remark}{section}
\newcommand{\nn}{{\mathbb{N}}}   % natuerliche Zahlen
\newcommand{\qq}{{\mathbb{Q}}}   % rationale Zahlen
\newcommand{\zz}{{\mathbb{Z}}}   % ganze Zahlen
\begin{document}
\title{Non-parametricity of rational translates of regular Galois extensions}
\author{Joachim K\"onig\thanks{Technion I.I.T. 32000 Haifa, Israel.
 email: koenig.joach@tx.technion.ac.il}}
\maketitle
\begin{abstract}
We generalize a result of F.\ Legrand about the existence of non-parametric Galois extensions for a given group $G$. More precisely, for a $K$-regular Galois extension $F|K(t)$, we consider the translates $F(s)|K(s)$ by an extension $K(s)|K(t)$ of rational function fields (in other words, $s$ is a root of $g(X)-t$ for some rational function $g\in K(X)$).
We then show that if $F|K(t)$ is a $K$-regular Galois extension with group $G$ over a number field $K$, 
then for any degree $k\ge 2$ and almost all (in a density sense) rational functions $g$ of degree $k$, the translate of $F$ by a root field of $g(X)-t$ over $K(t)$ is non-$G$-parametric, i.e.\ not all Galois extensions of $K$ with group $G$ arise as specializations of $F(s)|K(s)$.
\end{abstract}
\section{Introduction}
Let $K$ be a field. A Galois extension $F|K(t)$ is called $K$-regular (in the following simply {\textit regular}), if $F\cap \overline{K} = K$. For any $t_0\in K\cup\{\infty\}$ and any place $\mathfrak{p}$ of $F$ extending the $K$-rational place $t\mapsto t_0$, we have a residue field extension $F_{t_0}|K$. This is a Galois extension, not depending on the choice of place $\mathfrak{p}$. We call it the specialization of $F|K(t)$ at $t_0$.

Now let $G$ be a finite group. A Galois extension of a field $K$ with Galois group $G$ will be called a $G$-extension for short.
For $K$ a number field (or more generally a Hilbertian field), Hilbert's irreducibility theorem famously asserts that, given a regular $G$-extension $F|K(t)$, there are infinitely many $t_0\in K$ such that $F_{t_0}|K$ has the same Galois group $G$. A natural question is whether all $G$-extensions of $K$ arise in this way. This question can be made precise in several different ways: Firstly, the Beckmann-Black problem, first posed in \cite{Beckmann94} (for $K=\qq$), asks whether every Galois extension of $K$ with group $G$ is
the specialization of {\textit some} regular Galois extension with group $G$. This problem remains open (over number fields) for many groups $G$, and there is no group for which a negative answer is known. Beckmann showed in \cite{Beckmann94} that the answer is positive for abelian groups and symmetric groups; further examples were given by Black, for example for many dihedral groups in \cite{Black}.

One may further ask how many regular $G$-extensions are necessary to cover all $G$-extensions of $K$. This leads to the concept of $G$-parametric Galois extensions, introduced by Legrand in \cite{Legrand1}.
\begin{defi}[$G$-parametric Galois extension]
Let $K$ be a field, $F|K(t)$ be a regular Galois extension with group $G$. The extension $F|K(t)$ is called $G$-parametric (over $K$) if every Galois extension of $K$ with group $G$ arises as a specialization of $F|K(t)$. 
\end{defi}
Obviously, the existence of a parametric extension is sufficient, though not necessary, for a positive answer to the Beckmann-Black problem for the group $G$.
\section{Background on parametric extensions and statement of the Main Theorem}
In this paper, we focus on the case that $K$ is a number field. 
Very few $G$-parametric extensions over number fields are actually known. In particular, over $\mathbb{Q}$, only the subgroups of $S_3$ are known to possess a $G$-parametric extension.
On the other hand, it is quite difficult to show non-parametricity for a given regular extension, or even to show that there are any non-parametric extensions at all for a given group.
The last problem was solved over arbitrary number fields by Legrand in \cite{Legrand2}, where the following is shown:
\begin{satz}[Legrand]
Let $K$ be a number field, $F|K(t)$ a regular Galois extension with group $G$. Then there are infinitely many $k\in \nn$ with the following property:

The extension $F(\sqrt[k]{t})|K(\sqrt[k]{t})$ is a non-$G$-parametric regular Galois extension with group $G$.
More precisely, there are infinitely many specializations of $F|K(t)$ which are not specializations of $F(\sqrt[k]{t})|K(\sqrt[k]{t})$. 
\end{satz}

The aim of this article is to sharpen this result,
by proving that in fact, {\textit almost all} rational functions $g$ of a fixed degree $k$, instead of only functions of the form $g=X^k$ (and only for {\textit some} $k$), yield non-parametric extensions in the same way. 

By a rational translate of a Galois extension $F|K(t)$ we mean an extension $F(s)|K(s)$, where $K(s)|K(t)$ is an extension of rational function fields, i.e.\ $s$ is a root of $g(X)-t$ for some rational function $g\in K(X)$. 
%Note that, if $F|K(t)$ is the splitting field of a polynomial $p(t,X)\in K(t)[X]$, then $F(s)|K(s)$ is isomorphic to the splitting field of $p(g(t),X)$ over $K(t)$.

\begin{satz}[Main Theorem]
\label{almostall}
Let $K$ be a number field with ring of integers $\mathcal{O}_K$, let $F|K(t)$ and $F_2|K(t)$ be (not necessarily distinct) regular Galois extensions with group $G$, and let $k\ge 2$. 
Then for almost all polynomials $g_1,g_2\in \mathcal{O}_K[X]$ of degree $k$, the rational translate $F(s)|K(s)$, where $s$ is a root of $t-g(X):=t-g_1(X)/g_2(X)$, is a regular Galois extension with group $G$ fulfilling the following:\\ 
There are infinitely many Galois extensions of $K$ with group $G$ which arise as specializations of $F_2|K(t)$, but not as specializations of $F(s)|K(s)$. In particular, $F(s)|K(s)$ is non-$G$-parametric.
\end{satz}
Here, ``almost all" is to be understood in the sense of ``density $1$", as explained in Definition \ref{height} below.

An application of Theorem \ref{almostall} is the broader question of when two regular Galois extensions possess the same set of specializations, showing that this almost never happens among rational translates. This is contained in Section \ref{spec_equiv}.

\begin{remark}
\label{counterexample}
The following example shows that there do exist regular Galois extensions $F|K(t)$ and non-trivial rational translates which yield the same set of specializations as $F|K(t)$: The extension $K(\sqrt{t})|K(t)$ is $C_2$-parametric, and therefore, every quadratic extension $F|K(t)$ with exactly two branch points, both $K$-rational, is also $C_2$-parametric (since it can be transformed into the above extension by fractional linear transformations in $t$). Now let $s$ be such that $t=s^2-1$, then the translate of $K(\sqrt{t})|K(t)$ by $K(s)$ is the splitting field of $X^2-(s^2-1)$ over $K(s)$, which is quadratic with branch point set $\{\pm 1\}$, and therefore parametric by the above.
%\\
\end{remark}
\section{Auxiliary results}
In the following, $K$ always denotes a number field and $\mathcal{O}_K$ its ring of integers.
\subsection{A non-parametricity criterion}
Legrand exhibits several sufficient criteria for a regular $G$-extension to be non-$G$-parametric. The one that we will make use of in this paper uses the mod-$p$ behaviour of minimal polynomials of the branch points of a given regular extension:
\begin{defi}
\label{rampol}
Let $F|K(t)$ be a regular Galois extension, with branch points $t_1,...,t_r \in \overline{\qq}\cup\{\infty\}$. Define the ramification polynomial of $F|K(t)$ (with respect to $t$) as the homogeneous polynomial $\prod_{i=1}^r\mu_{t_i}(X,Y)$, where
$$\mu_{t_i}(X,Y):=\begin{cases}X-t_iY, \text{ if } t_i\ne \infty\\ 
Y, \text{ if } t_i=\infty\end{cases}.$$
\end{defi}
\begin{remark}\ 
\begin{itemize}
\item[a)] As branch points come in sets of algebraic conjugates, the ramification polynomial is in $K[X,Y]$, and equals the product of all (homogenized) minimal polynomials of branch points (without multiplicities) over $K$.
\item[b)] We deliberately work with the homogeneous setup in order not to get exceptions for the branch point $t\mapsto \infty$ in the following. Of course one can also always reduce to the case that $\infty$ is not a branch point by suitable fractional linear transformations in $t$. The subtle problem with this is that we want to look at rational translates given by $g(s)=t$, where we count rational functions $g$ up to a given height (see Def. \ref{height}). Transformations in $t$ would change the rational function $g$ and in particular distort the height.
\end{itemize}
\end{remark}
Legrand gives the following criterion (see \cite{Legrand1}, Theorem 4.2):\footnote{Note that the somewhat more convoluted definition of the ramification polynomial in \cite[Section 4.1.1]{Legrand1}, including also the minimal polynomials $\mu_{1/t_i}(X)$ of inverses of branch points, is unnecessary here, since $\mu_{t_i}$ and $\mu_{1/t_i}$ have the same splitting behaviour modulo almost all primes.}
\begin{prop}
\label{legrand_crit}
Let $F_1|K(t)$ and $F_2|K(t)$ be two regular Galois extensions with group $G$ with ramification polynomials $m_1$ and $m_2$. Assume that there are infinitely many primes $p$ of $\mathcal{O}_K$ such that $m_1$ has a root modulo $p$ but $m_2$ does not.
Then there are infinitely many specializations of $F_1|K(t)$ with group $G$ which do not arise as specializations of $F_2|K(t)$. In particular, $F_2|K(t)$ is non-$G$-parametric.
\end{prop}

\subsection{Height and density}
By a prime divisor of a polynomial $f$ we mean a prime $p$ such that
$f$ has a root modulo $p$. The main idea for the proof of our Main Theorem \ref{almostall} is now that a composition $f\circ g$ of polynomials almost never has the same set of prime divisors as the polynomial $f$. This is the combined content of Lemmas \ref{fullwr} and \ref{modp}, generalizing the results of \cite{Legrand2}.

To turn this into a precise statement, we first need a notion of height for polynomials over the algebraic integers of some number field. There are several ways to do this, but the following may be most convenient (see Section 2 of \cite{Cohen81}).\footnote{Since statements used in this paper, such as Hilbert's irreducibility theorem, hold for other notions of height (e.g.\ the logarithmic Weil height), it should not be difficult to regain our results for those height functions as well.}
\begin{defi}[Height of an algebraic integer]
Let $K$ be a number field with ring of integers $\mathcal{O}_K$. Let $\omega_1,...,\omega_n$ be an integral basis of $\mathcal{O}_K$ over $\mathbb{Q}$.
For $\alpha=\sum_{i=1}^n a_i \omega_i \in \mathcal{O}_K$ (with $a_i \in \zz$) define the height $H(\alpha)$ as $(\max_i |a_i|)^n$.
\end{defi}

We derive a notion of height for polynomials and rational functions. Note that the number field $K$ is always assumed to come with a fixed integral basis.

\begin{defi} [Height of a polynomial/ Density]
\label{height}
Let $K$ be a number field with ring of integers $\mathcal{O}_K$ and let $V_n:=\mathcal{O}_K[X]_{\le n}$ be the space of polynomials of degree $\le n$ over $\mathcal{O}_K$. For $f = \sum_{i=0}^n\alpha_iX^i \in V_n$, we define the height of $f$ as $H(f):=\max_{i=0}^n H(\alpha_i)$.\\
We say that a subset $S\subset V_n$ has density $d\in [0,1]$ if the limit 
$$\lim_{H\to \infty}\frac{|S\cap \{f\in V_n\mid H(f)\le H\}|}{|\{f\in V_n\mid H(f)\le H\}|}$$ exists and equals $d$.\\
In particular, we say that a property holds for almost all $f\in V_n$, if the density of the set of polynomials fulfilling this property is $1$.\\
The same notions will be used for $f=f_1+tf_2 \in V_n + t V_n$ (with a transcendental $t$), e.g.\ $H(f):=\max\{H(f_1), H(f_2)\}$. In slight abuse of notation, we will say that a property holds for almost all degree-$n$ rational functions over $\mathcal{O}_K$, if it holds for almost all $f=f_1+tf_2$ as above.
\end{defi}

We need a sufficiently strong version of Hilbert's irreducibility theorem over arbitrary number fields, considering integer specializations up to a given height. The following is contained in \cite{Cohen81}.
\begin{satz}[Cohen]
\label{hilbert_cohen}
Let $K|k$ be an extension of number fields, let $t:=(t_1,...,t_s)$ and let $f(X,t)\in K[t,X]$ be a non-zero polynomial in the indeterminates $X$ and $t_1,...,t_s$ (with $s\ge 1$). %X mehrdimensional nötig? 
Then for $N$ sufficiently large, the number of integer specializations $t\mapsto (\alpha_1,...,\alpha_s) \in (\mathcal{O}_k)^s$ such that $H(\alpha_i)\le N$ for all $i=1,..,s$ and such that
the Galois group of $f(X,\alpha)$ over $K$ does not equal the Galois group of $f(X,t)$ over $K(t)$, is at most
$$c\cdot N^{s-1/2}\cdot \log(N),$$ for some constant $c$ only depending on $s$, $K$ and $f$.
\end{satz}
\section{Proof of the Main Theorem}
We start with an easy observation. This and some variants (with straightforward modifications to the proof) will be used several times in the proofs of later statements. 
%Its proof essentially comes down to the following completely elementary fact:\\
%Let $K|k$ be a number field extension, $\alpha:=(\alpha_1,...,\alpha_r)$ with independent transcendentals $\alpha_i$, and let $0\ne f(\alpha)\in \mathcal{O}_K[\alpha]$. Then $f(\alpha_0)\ne 0$ for almost all (in the sense of density $1$) $\alpha_0\in (\mathcal{O}_k)^r$.
%
\begin{lemma}
\label{techn_lemma}
Let $S \subset \overline{\qq}\cup\{\infty\}$ be a finite subset, and let $k \in \nn$. For a degree-$k$ rational function $g$ over $K$, let $K(s)|K(t)$ be a root field of $g(X)-t$. Then for almost all degree-$k$ rational functions $g$ (in the sense of density $1$), no branch point of $K(s)|K(t)$ lies in $S$.
\end{lemma}
\begin{proof}
Firstly, we can assume that the denominator of $g$ is separable, so $t\mapsto \infty$ is not a branch point.\\
Let $g_1:=\sum_{i=0}^k \alpha_i X^i$ and $g_2:=\sum_{j=0}^k \beta_j X^j$ be generic polynomials of degree $k$ (with independent transcendentals $\alpha_i,\beta_j$. Write ${\alpha}:=(\alpha_0,...,\alpha_k)$ and ${\beta}:=(\beta_0,...,\beta_k)$. Let $\Delta$ be the discriminant of $g_1-tg_2$. After multiplying with a suitable factor, we can assume $\Delta\in \mathcal{O_K}[{\alpha},{\beta},t]$. The branch points of $K(s)|K(t)$ are just the roots of $\Delta({\alpha}_0,{\beta}_0,t)$ (with some specialization ${\alpha} \to {\alpha}_0$, ${\beta}\to {\beta}_0$).
Let $R_1,...,R_n$ be the roots of $\Delta$ in $\overline{K({\alpha},{\beta})}$. Certainly no $R_i$ is contained in $S$ (otherwise, every $K(s)|K(t)$ would have a branch point at this $R_i$).
Thus $\prod_{\stackrel{1\le i\le n}{t_0\in S}}(R_i-t_0) \in K(S)({\alpha},{\beta})$ is non-zero, and therefore remains non-zero under almost all specializations of ${\alpha},{\beta}$ in $\mathcal{O}_K$. This shows the assertion.
\end{proof}

In the following, we will consider compositions $f\circ g$ of a homogeneous polynomial $f\in \mathcal{O}_K[X,Y]$ with a rational function $g=g_1/g_2$ ($g_i\in \mathcal{O}_K[X]$). By this, we mean $f(g_1(X),g_2(X))$. Of course, if $Y$ does not divide $f$, this is just the same as the numerator of $f(g_1/g_2,1)$.
%
%\begin{lemma}
%\label{fullwr}
%Let $f\in \zz[X]$ be a non-constant irreducible polynomial with Galois group $G$, and let $k\ge 2$. Then for almost all $g\in \zz[X]$,
%%oder allgemeiner für rationale Funktionen...
%it holds that $Gal(f\circ g) \cong S_k \wr G$ is the full wreath product of $S_k$ with $G$. 
%%Besser: nicht nur für irreduzible Polynome formulieren?!?
%\end{lemma}
%\begin{proof}
%Obviously $Gal(f\circ g)$ is a subgroup of $S_k\wr G$ for all $g$ of degree $k$. 
%\end{proof}

%Stärkere Version:
\begin{lemma}
\label{fullwr}
Let $f\in \mathcal{O}_K[X,Y]$ be homogeneous and separable of degree $d\ge 1$, and let $L|K$ be a finite extension containing a splitting field of $f$. Let $k\ge 2$. Then for almost all rational functions $g$ over $\mathcal{O}_K$ of degree $k$, the Galois group of $f\circ g$ over $L$ is isomorphic to $(S_k)^d$, in the natural intransitive action with $d$ orbits of length $k$.
\end{lemma}
\begin{proof}
Let $\lambda_1(X,Y),...,\lambda_d(X,Y) \in L[X,Y]$ be the (homogeneous) linear factors of $f$, i.e.
$\lambda_i=X-t_iY$ for some $t_i\in L$, or $\lambda_i=Y$.
 
Set $G_1:= (\sum_{i=1}^{k} \alpha_i X^i) - t$, $\tilde{G_1}:=G_1+t$ and $G_2:=\sum_{j=0}^{k}\beta_j X_j$, with independent transcendentals $\alpha_i,\beta_j$ and $t$ (i.e.\ $G_1$ and $G_2$ are generic polynomials of degree $k$, only the constant coefficient of $G_1$ has been named separately because of the following treatment). 
Let $G=G_1/G_2$, and set $\alpha:=(\alpha_1,...,\alpha_k)$, $\beta:=(\beta_0,...,\beta_{k})$.
The polynomial $f\circ G$ then factors over $L(\alpha,\beta)$ as $\prod_{i=1}^d \lambda_i(G_1,G_2)$.

Now observe the polynomials $P_i:= \lambda_i(G_1,G_2) = \lambda_i(\tilde{G_1}-t, G_2)$, $i=1,...,d$.
These polynomials all have Galois group $S_k$ over $L(\alpha,\beta,t)$, since even the specialization $\beta_0=...=\beta_k=0$ (or $\alpha_1=...=\alpha_k=t=0$ in case $\lambda_i=Y$) leaves a generic degree-$k$-polynomial. 

Next we show that for all $i$ the splitting field $E_i$ of $P_i$ is linearly disjoint over $L(\alpha,\beta,t)$ to the composite of all other $E_j$; in other words, 
$$Gal(f\circ G|L(\alpha,\beta,t)) = Gal(\prod_{i=1}^d P_i |L(\alpha,\beta,t)) = (S_k)^d.$$ This is certainly true, if it holds for some specialization of the $\alpha$ and $\beta$.
First, specializing $\alpha_j\mapsto c\cdot \beta_j$ for all $j=1,...,k$ (with some fixed $c\in L)$ 
maps $P_i$ to $\underbrace{(c-t_i)}_{=:\mu_i} G_2 -\underbrace{(t+c)}_{=:s}$ (in case $\lambda_i\ne Y$)
or to $G_2$ (in case $\lambda_i=Y$).\\
Choosing $c$ appropriately, we may assume that none of the $\mu_i$ are zero.\\
Since the polynomial $G_2$ is generic, one shows as in Lemma \ref{techn_lemma} that almost all specializations of $\beta_0,...,\beta_k$ in $\mathcal{O}_L$ have the property that no two finite branch points of $g_2-s$ ($g_2$ being the specialization of $G_2$), viewed as a polynomial over $L(s)=L(t)$, differ by any ratio $\frac{\mu_j}{\mu_k}$ (for $1\le j\ne k\le d$). This however means that no two of the polynomials
$\mu_i g_2 - s$ have a common finite branch point.\\
Furthermore, for almost all such specializations, $\mu_i g_2-s$ will have squarefree discriminant and so only has simple branch points apart from infinity (i.e.\ inertia group generated by a transposition), which in particular means that the splitting fields of the $\mu_i g_2 -s$ are regular $S_k$-extensions of $L(s)$.
 Now since there is no non-trivial regular extension of $\overline{L}(s)$ ramified only at $s=\infty$, the splitting fields of the $\mu_i g_2-s$ must be linearly disjoint even over $\overline{L}$ (as the sets of their finite branch points are disjoint), and so their composite is still a regular extension of $L(s)$. In addition, the splitting field of the polynomial $g_2$ (which occurs if the linear factor $\lambda_i=Y$ occurs) is linearly disjoint over $L$ to the composite of all the other splitting fields, since it is a constant field extension. We have therefore shown $Gal(\prod_{i=1}^d P_i |L(\alpha,\beta,t)) = (S_k)^d$. 

Finally, by Hilbert's irreducibility theorem (in particular, the version in Theorem \ref{hilbert_cohen}) almost all specializations of the $\alpha_i$, $\beta_j$ and $t$
to values in $\mathcal{O}_K$ leave the Galois group over $L$ invariant, i.e.\ $Gal(f\circ g|L) = (S_k)^d$
for almost all degree-$k$ rational functions $g=g_1/g_2$ with $g_1,g_2 \in \mathcal{O}_K[X]$. 
\end{proof}

\begin{lemma}
\label{modp}
Let $L$, $f$ and $g$ be as in Lemma \ref{fullwr}, i.e.\ such that $Gal(f\circ g|L) =S_k^d$. Let $f_2 \in \mathcal{O}[X,Y]$ be homogeneous and completely split over $L$.\footnote{Note that this includes of course the special case $f_2=f$.} Then there is a positive density set of primes $p$ of $\mathcal{O}_K$ such that $f_2$ splits completely modulo $p$, but $f\circ g$ does not have a root modulo $p$.
\end{lemma}
\begin{proof}
Let $G:=Gal(f\circ g | K)$. 
Obviously, the normal subgroup $S_k^d=Gal(f\circ g|L)$ of $G$ contains an element $\sigma$ acting fixed point freely on the roots of $f\circ g$, namely every element which acts fixed point freely on each of the $d$ orbits of $S_k^d$. At the same time, $\sigma$ of course fixes all roots of $f_2$. By the Chebotarev density theorem, there is a positive density set of primes $p$ with Frobenius element conjugate to $\sigma$ in $G$.
But the cycle type of $\sigma$ in the action on the roots of $f\circ g$ corresponds to the splitting behavior of $f\circ g$ modulo $p$, while the cycle type in the action on the roots of $f_2$ corresponds to the splitting behavior of $f_2$.
Therefore, for all such $p$, $f_2$ splits completely modulo $p$, while $f\circ g$ does not have a root.
%acts imprimitively on the roots of $f\circ g$, with $d$ blocks of size $k$. Since $k>1$, there is an element $\sigma\in S_k \wr G$ stabilizing a block, but not a point (even the simultaneous stabilizer
%of all blocks, $S_k^r$, has fixed point free elements). By the Chebotarev density theorem, there is a positive density set of primes $p$ with Frobenius element conjugate to $\sigma$ in $S_k \wr G$. But the cycle type of $\sigma$
%corresponds to the splitting behaviour of $f\circ g$ modulo $p$, while the cycle type of its image in the action on the blocks corresponds to the splitting behaviour of $f$.
%Therefore, for all such $p$, $f$ has a root modulo $p$, while $f\circ g$ does not.
\end{proof}

%\section{Proof of the Main Theorem}
Now we are ready for the proof of Theorem \ref{almostall}. 
\begin{proof}
Let $S$ be the set of branch points of the extension $F|K(t)$. Firstly, by Lemma \ref{techn_lemma}, for almost all rational functions $g$ of a fixed degree, the ramification loci of $K(s)|K(t)$ and of $F|K(t)$ are disjoint (where $s$ is such that $g(s)=t$). This condition already forces $F$ and the Galois closure of $K(s)|K(t)$ to be linearly disjoint even over $\overline{K}(t)$, and is therefore in particular sufficient to ensure that $F(s)|K(s)$ is still $K$-regular with Galois group $G$.
Also, by Abhyankar's Lemma (e.g.\ \cite[Theorem 3.9.1]{St}), disjointness of the ramification loci implies that ramification indices in $F(s)|K(t)$ are the same as in $F|K(t)$ (for primes ramifying in $F|K(t)$) or as in $K(s)|K(t)$ (for primes ramifying in $K(s)|K(t)$).
The primes of $K(s)$ ramifying in $F(s)$ are then exactly the ones extending primes of $K(t)$ ramifying in $F$. After constant field extension from $K$ to $\overline{K}$ the ramified primes of $F|K(t)$ split exactly into the $t\mapsto t_i$, with $t_i\in S$. In the same way, the primes of $K(s)$ extending these primes split into the $s\mapsto s_i$ with $s_i\in g^{-1}(S)$. In other words, the branch point set of $F(s)|K(s)$ is just the preimage $g^{-1}(S)$.

Now let $f$ be the ramification polynomial of $F|K(t)$, multiplied by a suitable constant to make its coefficients integral. By definition, the ramification polynomial of $F(s)|K(s)$ then equals $f\circ \tilde{g}$ (up to possibly multiplicity of roots), where $\tilde{g}$ is the homogenization of (numerator and denominator of) $g$. For purposes of splitting, we may simply identify this with $f\circ g$.

Similarly, let $f_2$ be the ramification polynomial of $F_2|K(t)$, and let $L$ be the splitting field of $f\cdot f_2$ over $K$.
 We know from Lemma \ref{fullwr} that for almost all $g$, the Galois group of $f\circ g$ over $L$ is isomorphic to $(S_k)^d$ (where of course $k$ and $d$ are the degrees of $g$ and $f$). 

We know from Lemma \ref{modp} that there exist infinitely many primes such that $f\circ g$ has no root mod $p$ whereas $f_2$ splits completely. The assertion now follows immediately from Legrand's criterion (Prop.\ \ref{legrand_crit}).
 %Referenzieren.
\end{proof}

\section{A generalization: Specialization-equivalence of Galois extensions}
\label{spec_equiv}
Here we consider a notion of specialization-equivalence of Galois extensions, generalizing the problem of $G$-parametricity.\\
For any (not necessarily regular) Galois extension $F|K(t)$, let $\mathcal{S}_F$ be the set of all specializations $F_{t_0}|K$, where $F_{t_0}$ is obtained from $F$ by specializing $t\mapsto t_0\in K\cup\{\infty\}$. Call the extensions $F|K(t)$ and $F_2|K(t)$ {\textit specialization-equivalent}, if $\mathcal{S}_{F} = \mathcal{S}_{F_2}$.

The problem of specialization-equivalence is particularly interesting in the investigation of families of regular Galois extensions. It is a natural question to ask whether or not a member of such a family is uniquely characterized within the family by its set of specializations. One such family is the family of rational translates of a prescribed regular $G$-extension $F|K(t)$.\\
As a consequence of Theorem \ref{almostall}, specialization-equivalence is rare among rational translates. This is made precise in the following. To ease notation, for a non-constant rational function $g$ over $K$ we write $\mathcal{S}_{F,g}$ for the set of specializations of $F(s)|K(s)$, where $s$ is a root of $g(X)-t$. 
%
%
%%%%%
%NOTE: Same result (with same proof) remains true for ANY regular extension $L|K(t)$, compared to the
%degree-k translates of some given $F|K(t)$!!!
%
\begin{kor}
\label{unique_spez}
Let $F|K(t)$ and $F_2|K(t)$ be regular Galois extensions with group $G$, let $g_2$ be any non-constant rational function over $K$, and let $k\ge 2$. Then for almost all degree-$k$ rational functions $g$ over $K$ the following holds:\\
%REGULARITY ASSUMPTION NECESSARY!
The sets $\mathcal{S}_{F,g}$ and $\mathcal{S}_{F_2,g_2}$ differ by infinitely many elements. In particular, the translates of $F$ by a root field of $g(X)-t$ and of $F_2$ by a root field of $g_2(X)-t$ are not specialization-equivalent.
\end{kor}
\begin{proof}
Let $s_2$ be a root of $g_2(X)-t$.
We may assume that $F_2(s_2)|K(s_2)$ still has Galois group $G$; otherwise $\mathcal{S}_{F_2,g_2}$ does not contain any $G$-extensions, whereas almost all degree-$k$ rational functions $g$ lead to extensions still with Galois group $G$ (see the proof of the Main Theorem) and therefore to infinitely many specializations with group $G$.
With a similar argument, we can assume that $F_2(s_2)|K(s_2)$ is regular (otherwise all its specializations would contain the same non-trivial subextension, whereas almost all $g$ lead to regular extensions, which then possess infinitely many linearly disjoint specializations). 

Now the assertion follows immediately from Theorem \ref{almostall}.
\end{proof}
\begin{remark}
As pointed out by the referee, the example of Remark \ref{counterexample}, yielding specialization-equivalent members in a family of rational translates, is due to the fact that there, the extension $K(\sqrt{t})|K(t)$ and its translate by $K(s)$ are isomorphic. It is natural to ask whether every such example arises in this way. In fact, this is not the case: One can construct (however, by means exceeding the scope of this paper) examples of regular Galois extensions $F|K(t)$ and translates $F(s)|K(s)$ such that $F|K(t)$ is itself isomorphic to a rational translate of $F(s)|K(s)$, but $F|K(t)$ and $F(s)|K(s)$ are {\textit not} isomorphic. The classification of such exceptional cases seems to be an interesting subject for further research.
\end{remark}

\section{Concluding remarks}
While we have shown that for any given group $G$ and any number field $K$, non-$G$-parametric extensions are abundant (under the trivially necessary condition that there are regular $G$-extensions over $K$ at all), it is still left open whether there are groups $G$ such that no $G$-parametric extension exists over any number field $K$. The first examples of such $G$ are given in joint work by the author and F.\ Legrand (\cite{KL}). It should however be emphasized that the methods used there can never yield results for all finite groups $G$ (in particularly not for simple groups).

{\textbf Acknowledgment}\\
I thank the referee for several helpful suggestions.\\
This work was partially supported by the Israel Science Foundation (grant no. 577/15).

\end{document}